\documentclass[12pt]{amsart}
\usepackage{amsmath, amssymb, latexsym, amsthm,bm}
\usepackage{mathtools}
\usepackage{url}
\usepackage[T1]{fontenc}
\usepackage{mathrsfs}
\usepackage{color}
\usepackage{siunitx}
\usepackage{tikz,tikz-3dplot}
\usepackage{graphicx}
\usepackage{stmaryrd,amsbsy}
\usepackage{enumitem}
\usepackage{cite}

\def\Ls#1{{\text{\sc Link}(#1)}}

\makeatletter
\def\inmod#1{\allowbreak\mkern5mu({\operator@font mod}\,\,#1)}
\makeatother

\makeatletter
\@namedef{subjclassname@2020}{%
  \textup{2020} Mathematics Subject Classification}
\makeatother

\usepackage{enumerate}
\usepackage{xcolor}
\usepackage{hyperref}
\usepackage{pgfplots}
\usetikzlibrary{calc,3d,intersections, positioning,intersections,shapes}

\usepackage{mathscinet}
\usetikzlibrary{calc}
\usepackage{cancel}
\usepackage{xspace}

\usepackage[mathlines]{lineno}

\usepackage{url}
\usepackage{longtable}
\usepackage[obeyFinal,colorinlistoftodos,textsize=tiny]{todonotes}

\usepackage{tikz}

\usepackage[font=small]{caption}
\usepackage{mathtools,amsmath}
\usepackage{array}

\usepackage{tikz}

\usepackage{stmaryrd}

\def\sphere{{{\mathbb S}^2}}

\DeclarePairedDelimiterX{\infdivx}[2]{(}{)}{%
  #1\;\delimsize\|\;#2%
}

\renewcommand{\descriptionlabel}[1]%
  {{\hglue -0.7 cm}\hspace{\labelsep}#1}

\newcommand*{\Scale}[2][4]{\scalebox{#1}{\ensuremath{#2}}}%

\def\tf#1{{\Scale[2.4]{#1}}}
\def\tz#1{{\Scale[2.0]{#1}}}

\def\te#1{{\Scale[2.8]{#1}}}

\def\ignore#1{{ }}

\definecolor{mygray}{gray}{0.8}

\def\cc{{\mathscr C}}

\def\myfrac#1#2{{\genfrac{}{}{0pt}{}{#1}{#2}}}

\def\cir#1#2#3{{\cc{\hbox{\hglue #1 cm}\myfrac{#2}{#3}}}}

\def\centerarc[#1](#2)(#3:#4:#5){ \draw[#1] ($(#2)+({#5*cos(#3)},{#5*sin(#3)})$) arc (#3:#4:#5); } 

\def\conethree{\cir{-0.04}{1}{3}}

\def\cthreefive{\cir{-0.02}{3}{5}}

\def\c13{{\conethree}}
\def\c35{{\cthreefive}}

\def\oP{{\overline{P}}}
\def\oD{{\overline{D}}}
\def\oB{{\overline{B}}}

\newtheorem{theorem}{Theorem} 
\newtheorem{theorem*}{Theorem} 
\newtheorem{proposition}[theorem]{Proposition} 

\newtheorem{claim}[theorem]{Claim}

\newtheorem{lemma}[theorem]{Lemma}

\newtheorem{observation}[theorem]{Observation}

\begin{document}



\title[Regular projections of the link $L6n1$]{Regular projections of the link $L6n1$}


\author[Alba]{Andrea Alba}
\address{Instituto de F\'\i sica, Universidad Aut\'onoma de San Luis Potos\'{\i}, SLP 78000, Mexico}
\email{\tt andrea.casillas@if.uaslp.mx}

\author[Ram\'\i rez]{Santino Ram\'\i rez}
\address{Instituto de F\'\i sica, Universidad Aut\'onoma de San Luis Potos\'{\i}, SLP 78000, Mexico}
\email{\tt santinormz@if.uaslp.mx}

\author[Salazar]{Gelasio Salazar}
\address{Instituto de F\'\i sica, Universidad Aut\'onoma de San Luis Potos\'{\i}, SLP 78000, Mexico}
\email{\tt gsalazar@ifisica.uaslp.mx}



\date{\today}
\def\W#1{{\widehat{#1}}}
\def\solv#1#2{{{#1}\rightarrowtail{#2}}}


\begin{abstract}
Given a link projection $P$ and a link $L$, it is natural to ask whether it is possible that $P$ is a projection of $L$. Taniyama answered this question for the cases in which $L$ is a prime knot or link with crossing number at most five. Recently, Takimura settled the issue for the knot $6_2$. We answer this question for the case in which $L$ is the link $L6n1$.
\end{abstract}

\maketitle

\section{Introduction}\label{sec:intro}

We work in the piecewise linear category, and links are hosted in the $3$-sphere ${\mathbb S}^3$. If we project a link $L$ onto the $2$-sphere $\sphere$, we obtain a {\em projection} of $L$. We {\em resolve} a link projection by giving over/under information at each crossing, thus obtaining a link {\em diagram}. 

As usual, we are interested in {regular} projections of links, such as the projection $P$ in Figure~\ref{fig:100}. We recall that a projection is {\em regular} if it has finitely many multiple points, and each multiple point is a transverse (non-tangential) double point. Two projections are {\em equivalent} if there is a self-homeomorphism of $\sphere$ that takes one to the other.

Let $P$ be a projection and let $L$ be a link. We do not distinguish between isotopic links, and so $P$ {\em is a projection of $L$} if there is a link isotopic to $L$ that projects to $P$. Thus $P$ is a projection of $L$ if and only if it is possible to resolve $P$ to obtain a diagram $D$ of a link isotopic to $L$. See Figure~\ref{fig:100} for an illustration. 

\begin{figure}[ht!]
\centering
\scalebox{0.25}{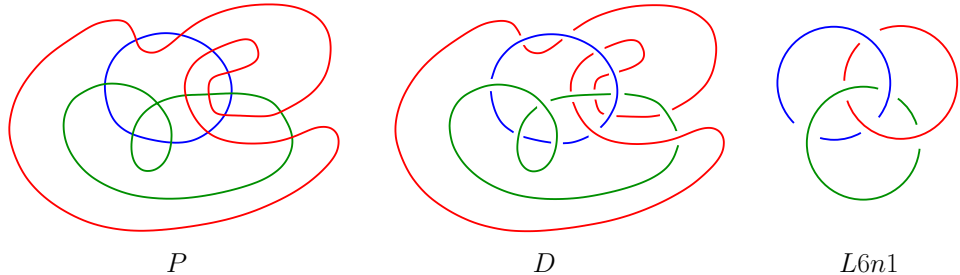}
\caption{On the left-hand side we have a projection $P$ of a $3$-component link. We turn $P$ into the diagram $D$ at the center of the figure by giving over/under information at each crossing, that is, by {\em resolving} each crossing of $P$. Using a sequence of Reidemeister moves, it is possible to transform $D$ into the diagram on the right-hand side, which is a diagram of the link $L6n1$. Therefore $P$ is a projection of $L6n1$.}
\label{fig:100}
\end{figure}

Given a projection $P$ and a link $L$, it is natural to ask whether $P$ is a projection of $L$. This question and several variants have been thoroughly investigated in the literature~\cite{cantarella,endoitoh,evenzohar,hanaki2009,hanaki2010,hanaki2014,hanaki2015,hanaki2020,huhtaniyama,itotakimura,medina1,smooth,ptaniyama}. Taniyama answered this question for all the prime links (including knots) with crossing number at most five~\cite{taniyamaknots,taniyamalinks}. Recently, Takimura settled the issue for the knot $6_2$~\cite{takimura2018}.

\subsection{Our main result} In this paper we answer this question for the case in which $L$ is a $3$-component prime link with crossing number six, namely the link $L6n1$ in the Thistlethwaite Link Table~\cite{atlas}. We refer the reader again to the right-hand side of Figure~\ref{fig:100}. 

In order for a projection $P$ to be a projection of $L6n1$, an obvious requirement is that $P$ be a projection of a $3$-component link. Every such projection $P$ is the union of three knot projections, which we colour arbitrarily so that we have a {\em blue} knot projection $B$, a {\em red} knot projection $R$, and a {\em green} knot projection $G$. We refer the reader again to Figure~\ref{fig:100}. 

We say that $P$ is {\em pairwise crossing} if $B,R$, and $G$ pairwise cross each other. Since the components of the link $L6n1$ are pairwise linked, it follows that if $P$ is a projection of $L6n1$, then $P$ must be pairwise crossing. Our main result is that this obvious necessary condition is also sufficient.

\begin{theorem}\label{thm:main}
A projection of a $3$-component link is a projection of the link $L6n1$ if and only if it is pairwise crossing.
\end{theorem}


\subsection{Overview of the proof of Theorem~\ref{thm:main}}

For convenience, for the rest of the paper we regard every link projection as a $4$-regular graph embedded on $\sphere$, by turning each crossing into a degree $4$ vertex. See Figure~\ref{fig:200}.

Following~\cite{ptaniyama} and~\cite{taniyamaknots}, if $P$ is a link projection then $\Ls{P}$ denotes the set of all those links $L$ such that $P$ is a projection of $L$. With this notation, Theorem~\ref{thm:main} claims that if $P$ is a $3$-component link projection, then $L6n1 \in \Ls{P}$ if and only if $P$ is pairwise crossing.


A key tool in the proof of Theorem~\ref{thm:main} is the identification of two {\em reduction operations} on link projections, introduced in Section~\ref{sec:red}. As we will see, when we perform a reduction operation on a pairwise crossing projection $P$ we obtain a projection $\oP$ that satisfies the following properties:

\vglue 0.1 cm
\noindent{(R1)} {\sl $\oP$ has fewer vertices than $P$.}
\vglue 0.1 cm
\noindent{(R2)} {\sl $\oP$ is pairwise crossing.}
\vglue 0.1 cm
\noindent{(R3)} {\sl $\Ls{\oP} \subseteq \Ls{P}$.}
\vglue 0.1 cm

A pairwise crossing projection $P$ is {\em irreducible} if we cannot apply any reduction operation to it. If we start with an arbitrary pairwise crossing projection $P$, after performing a finite number of reduction operations we end up with an irreducible projection $P'$. An iterative application of (R3) yields that $\Ls{P'} \subseteq \Ls{P}$. In particular, if $P'$ is a projection of $L6n1$, then $P$ is also a projection of $L6n1$. Therefore  {\em in order to prove Theorem~\ref{thm:main}, it suffices to show that every irreducible projection is a projection of $L6n1$}.


To achieve this goal, we fully characterize which projections are irreducible: as we state in Section~\ref{sec:proofmain} (see Proposition~\ref{pro:twoirr} and Figure~\ref{fig:670}) up to equivalence there are only two irreducible projections. It is easy to see that they are both projections of $L6n1$ (see Figure~\ref{fig:670}). Thus every irreducible projection is indeed a projection of $L6n1$, and so Theorem~\ref{thm:main} follows. 

\section{The reduction operations}\label{sec:red}

Before we identify the reduction operations, we make a few elementary remarks on link projections. Recall that we regard link projections as $4$-regular graphs embedded on $\sphere$.

\subsection{Straight-ahead walks, monochromatic vertices, and bichromatic vertices}
We recall that a {\em straight-ahead walk} in a $4$-regular graph is a walk in which every time we reach a vertex $v$ in the walk, the next edge in the walk is the opposite edge to the one from which we arrived to $v$. A walk is {\em closed} if its final vertex is the same as its initial vertex.

We are interested in projections of $3$-component links, and so every projection under consideration is a projection of a $3$-component link. Thus every projection $P$ of interest is the edge-disjoint union $P=B\cup R\cup G$ of three straight-ahead closed walks: a {\em blue} walk $B$, a {\em red} walk $R$, and a {\em green} walk $G$. We refer the reader to Figure~\ref{fig:200}. 

\begin{figure}[ht!]
\centering
\scalebox{0.31}{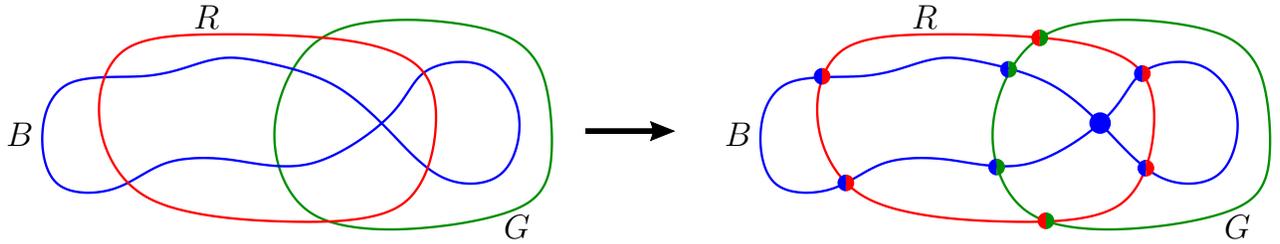}
\caption{A link projection is regarded as a $4$-regular graph embedded on $\sphere$, by turning each crossing into a degree $4$ vertex. Every $3$-component link projection is the edge-disjoint union of three straight-ahead closed walks: a {blue} walk $B$, a {red} walk $R$, and a {green} walk $G$. The projection in this figure is pairwise crossing: there is at least one blue-red vertex, at least one blue-green vertex, and at least one red-green vertex.}
\label{fig:200}
\end{figure}

If $v$ is a vertex of $P$, and the four edges incident with $v$ are of the same colour, then we also colour $v$ with this colour, and say that $v$ is {\em monochromatic}. Otherwise, two of the edges incident with $v$ are of one colour, and the other two are of another colour. In this case we colour $v$ with both colours, and say that $v$ is {\em bichromatic}. Thus a monochromatic vertex is either blue, red, or green, and there are three {\em types} of bichromatic vertices: blue-red, blue-green, or red-green. See Figure~\ref{fig:200}.


In view of Theorem~\ref{thm:main}, our interest lies in pairwise crossing projections. In every such projection the blue straight-ahead closed walk $B$ and the red straight-ahead closed walk $R$ must cross each other, and so there is at least one blue-red vertex. Actually, using the Jordan curve theorem it is easy to see that there must be at least two blue-red vertices. Similarly, there are at least two blue-green vertices and at least two red-green vertices. 

\subsection{The reduction operations}

As we mentioned in the previous section, an essential tool in the proof of Theorem~\ref{thm:main} is the existence of two {\em reduction operations} on link projections, each of which satisfies properties (R1)--(R3). 

\subsubsection{The first reduction operation: shortcutting a projection}

Let $P=B\cup R\cup G$ be a pairwise crossing projection. {Suppose that there is a face $f$ in $P$ whose boundary contains two edges $e,e'$ of the same colour, which without loss of generality we may assume to be blue. We refer the reader to Figure~\ref{fig:220}(i) for an illustration, where $f$ is the shaded face and $e, e'$ are the thick edges.}

{As illustrated in Figure~\ref{fig:220}(ii), we subdivide the edge $e$ with a degree $2$ vertex $x$, and subdivide $e'$ with a degree $2$ vertex $y$. As we show in Figure~\ref{fig:220}(iii) and (iv), the blue straight-ahead closed walk $B$ is naturally decomposed into two straight-ahead walks $B_1$ and $B_2$, each of which starts at $x$ and ends at $y$.} {We say that $B_1$ and $B_2$ are the $xy$-{\em walks}. For $i\in\{1,2\}$, we say that $B_i$ is {\em colourful} if it has at least one blue-red vertex and at least one blue-green vertex.}

\def\tq#1{{\Scale[2.0]{#1}}}
\def\tj#1{{\Scale[3.0]{#1}}}
\def\tz#1{{\Scale[2.4]{#1}}}
\begin{figure}[ht!]
\centering
\scalebox{0.29}{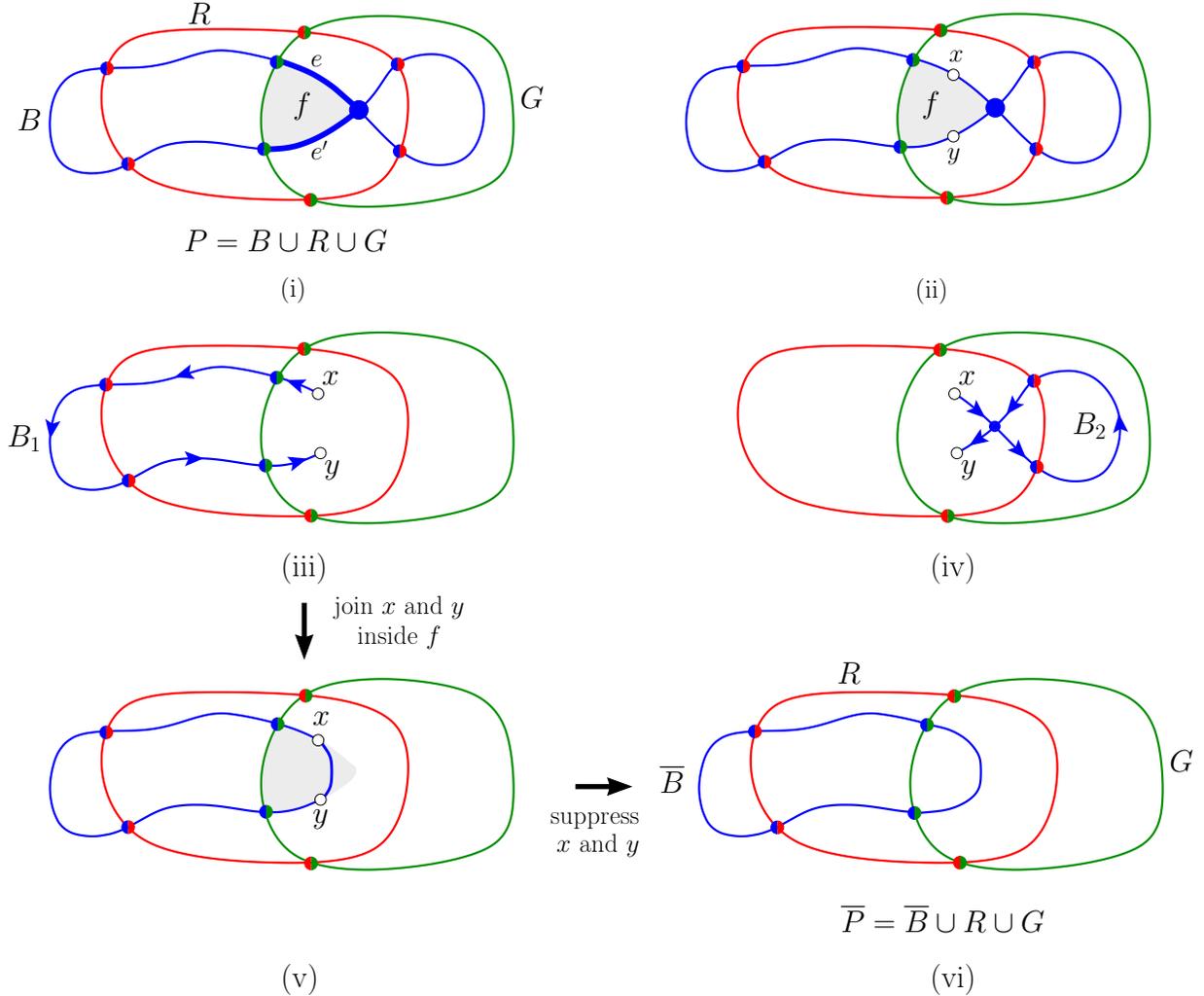}
\caption{Illustration of the shortcut operation.}
\label{fig:220}
\end{figure}
\def\tz#1{{\Scale[2.0]{#1}}}

{Suppose that at least one of $B_1$ and $B_2$ is colourful. As in the example in Figure~\ref{fig:220}, we may assume without loss of generality that $B_1$ is colourful. We can then proceed to {\em shortcut} $B$, as follows. First, we discard $B_2$ and join $x$ and $y$ in $B_1$ with an arc contained in $f$, as illustrated in Figure~\ref{fig:220}(v). Finally we suppress $x$ and $y$. As illustrated in Figure~\ref{fig:220}(vi), as a result we obtain a new blue straight-ahead closed walk $\overline{B}$. We say that to obtain $\overline{B}$ we {\em shortcut} $B$ {\em at $x$ and $y$}, and the projection $\oP=\oB\cup R\cup G$ is obtained by {\em shortcutting $P$ at $x$ and $y$}.}

In order to show that properties (R1)--(R3) hold, we start by noting that the internal vertices of the discarded $xy$-walk ($B_2$ in our previous discussion) are in $P$ but not in $\oP$. Thus (R1) holds. We also note that the colourfulness of $B_1$ implies that $\oP$ is pairwise crossing. Thus (R2) holds.

We finally move on to showing that (R3) holds. We illustrate the proof in Figure~\ref{fig:400}, using the projections $P$ and $\oP$ from Figure~\ref{fig:220}. Suppose that $L\in \Ls{\oP}$, that is, $\overline{P}$ is a projection of a link $L$. Thus it is possible to resolve each vertex of $\overline{P}$ so that the result is a diagram $\overline{D}$ of $L$. 

To prove (R3) we need to show that $P$ is also a projection of $L$. To achieve this, we describe how to resolve each vertex of $P$ to obtain a diagram $D$ equivalent to $\oD$. 

\def\tf#1{{\Scale[2.4]{#1}}}
\def\tz#1{{\Scale[2.0]{#1}}}
\begin{figure}[ht!]
\centering
\scalebox{0.295}{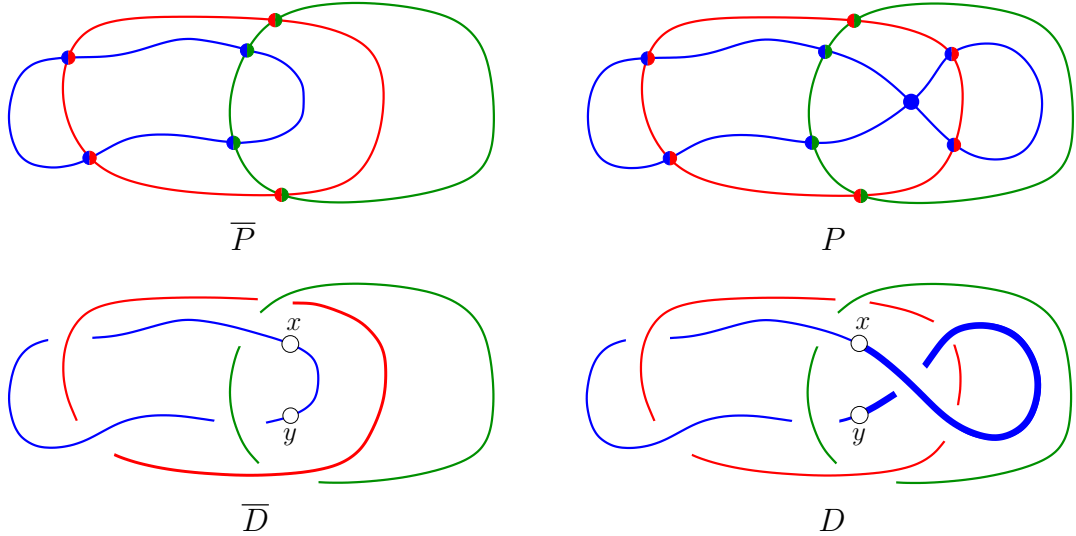}
\caption{Using the descending algorithm, a diagram $\oD$ of $\oP$ can be extended to a diagram $D$ of $P$, equivalent to $\oD$.}
\label{fig:400}
\end{figure}
\def\tf#1{{\Scale[2.4]{#1}}}
\def\tz#1{{\Scale[2.0]{#1}}}

First, for each vertex of $P$ that is also a vertex of $\overline{P}$, we resolve it exactly as it was resolved back in $\overline{P}$ in order to obtain $\overline{D}$. See Figure~\ref{fig:400}. 

We now describe how to resolve each vertex of $P$ that is not a vertex in $\overline{P}$. Note that these are precisely the vertices in the straight-ahead walk $B_2$ that was discarded in the shortcutting process. Recall that the endpoints of $B_2$ are $x$ and $y$. As in~\cite{ptaniyama} and~\cite{taniyamaknots}, we use the {\em descending algorithm} from $x$ to $y$ to resolve each vertex of $B_2$: we traverse $B_2$ from $x$ to $y$, and whenever we arrive to a vertex for the first time, we resolve this vertex so that the strand we are currently traversing is the overstrand.  We refer the reader again to Figure~\ref{fig:400}: the thick part is obtained using the descending algorithm from $x$ to $y$.

Let $D$ be the diagram obtained by resolving the vertices of $P$ in the way we have described. The use of the descending algorithm implies that the strand in $D$ from $x$ to $y$ can be isotoped to the strand in $\oD$ from $x$ to $y$. This implies that $D$ and $\oD$ are equivalent diagrams, and so the proof of (R3) is complete.

Needless to say, choosing the blue straight-ahead closed walk $B$ for the discussion was arbitrary, as evidently a totally analogous shortcut operation can be applied to $R$ or to $G$. 

\subsubsection{The second reduction operation: simplifying a $\Theta$}

Suppose that $P=B\cup R\cup G$ contains a straight-ahead cycle $C=uvwu$, such as the one illustrated in Figure~\ref{fig:1160}(i). Suppose that (a) $v$ and $w$ are joined by an edge $e$ distinct from the edge that joins them in $C$; and (b) the connected component of $\sphere\setminus C$ that contains $e$ does not contain any other part of $P$. We then say that $C+e$ is a {\em $\Theta$} in $P$. See Figure~\ref{fig:1160}(ii).

\def\te#1{{\Scale[2.6]{#1}}}
\def\tf#1{{\Scale[3.2]{#1}}}
\def\tz#1{{\Scale[3.0]{#1}}}
\begin{figure}[ht!]
\centering
\scalebox{0.25}{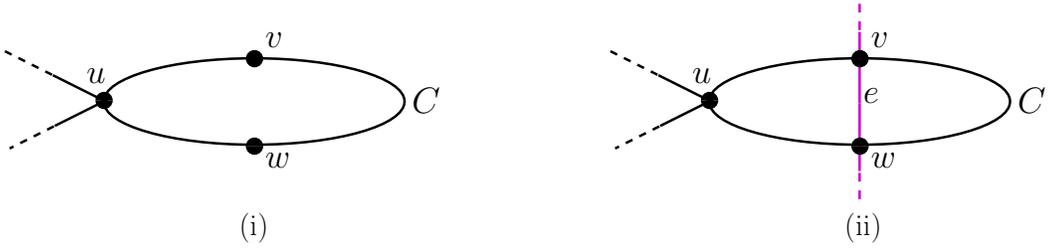}
\caption{Illustration of the notion of a $\Theta$ in a projection.}
\label{fig:1160}
\end{figure}
\def\tf#1{{\Scale[2.4]{#1}}}
\def\tz#1{{\Scale[2.0]{#1}}}
\def\te#1{{\Scale[3.4]{#1}}}

We {\em simplify} $\Theta$ by {\em splitting} $u$ as illustrated in Figure~\ref{fig:1620}. This is the second and last reduction operation on link projections that we will use in this paper.

\def\te#1{{\Scale[2.6]{#1}}}
\def\tf#1{{\Scale[3.2]{#1}}}
\def\tz#1{{\Scale[3.0]{#1}}}
\begin{figure}[ht!]
\centering
\scalebox{0.25}{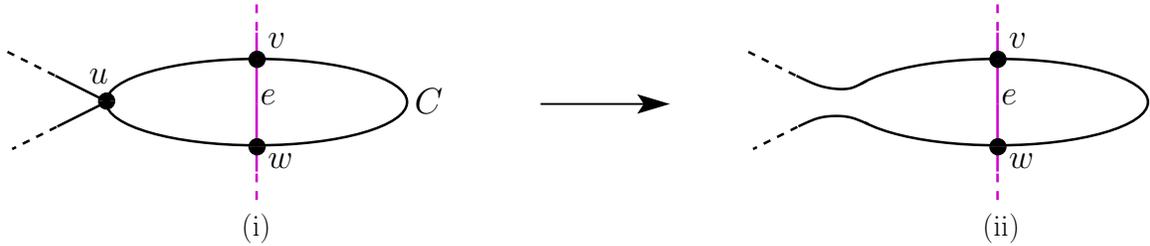}
\caption{To simplify the $\Theta$ on the left hand side, we open up (``split'') $u$ as shown on the right hand side.}
\label{fig:1620}
\end{figure}
\def\tf#1{{\Scale[2.4]{#1}}}
\def\tz#1{{\Scale[2.0]{#1}}}
\def\te#1{{\Scale[3.4]{#1}}}

The projection $\oP$ obtained from $P$ by simplifying the $\Theta$ has one fewer vertex than $P$, as $u$ belongs to $P$ but it is not in $\oP$ anymore. Thus (R1) holds. Now since $C$ is a straight-ahead cycle it follows that $u$ is a monochromatic vertex. Therefore every bichromatic vertex in $P$ is also in $\oP$, and since $P$ is pairwise crossing it follows that $\oP$ is also pairwise crossing. Therefore Property (R2) also holds for this reduction operation.

We now prove that (R3) holds. Suppose that $L\in \Ls{\oP}$, that is, $\oP$ is a projection of a link $L$. Thus it is possible to resolve each vertex of $\oP$ so that the result is a diagram $\oD$ of $L$.

To prove (R3) we need to show that $P$ is also a projection of $L$. To achieve this, we describe how to resolve each vertex of $P$ to obtain a diagram $D$ equivalent to $\oD$.

First, for each vertex of $P$ that is not in $\{u,v,w\}$, we resolve it exactly as it was resolved back in $\oP$ in order to obtain $\oD$. It remains to describe how to resolve $u,v$, and $w$. 

The way in which we resolve $u,v$, and $w$ in $P$ depends on how $v$ and $w$ are resolved in $\oP$ to obtain $\oD$. For instance, if $v$ and $w$ are resolved in $\oP$ as in Figure~\ref{fig:1890}(i), then we resolve $u,v$, and $w$ in $P$ as in Figure~\ref{fig:1890}(ii). It is easy to see that the strand from $x$ to $y$ in (ii) (that is, in $D$) can be isotoped to the strand from $x$ to $y$ in (i) (that is, in $\oD$). Since all the other vertices of $P$ are resolved in the same way as they are resolved in $\oP$, we conclude that the resulting diagram $D$ of $P$ is equivalent to the diagram $\oD$ of $\oP$. 

There are three more ways in which the vertices $v$ and $w$ can be resolved in $\oP$. These are illustrated in Figure~\ref{fig:1890}(iii), (v), and (vii). If $v$ and $w$ are resolved in $\oP$ as in (iii), then we resolve $u,v$, and $w$ in $P$ as in (iv). If $v$ and $w$ are resolved in $\oP$ as in (v), then we resolve $u,v$, and $w$ in $P$ as in (vi). Finally, if $v$ and $w$ are resolved in $\oP$ as in (vii), then we resolve $u,v$, and $w$ in $P$ as in (viii).

\def\te#1{{\Scale[2.6]{#1}}}
\def\tf#1{{\Scale[3.2]{#1}}}
\def\tz#1{{\Scale[3.0]{#1}}}
\begin{figure}[ht!]
\centering
\scalebox{0.25}{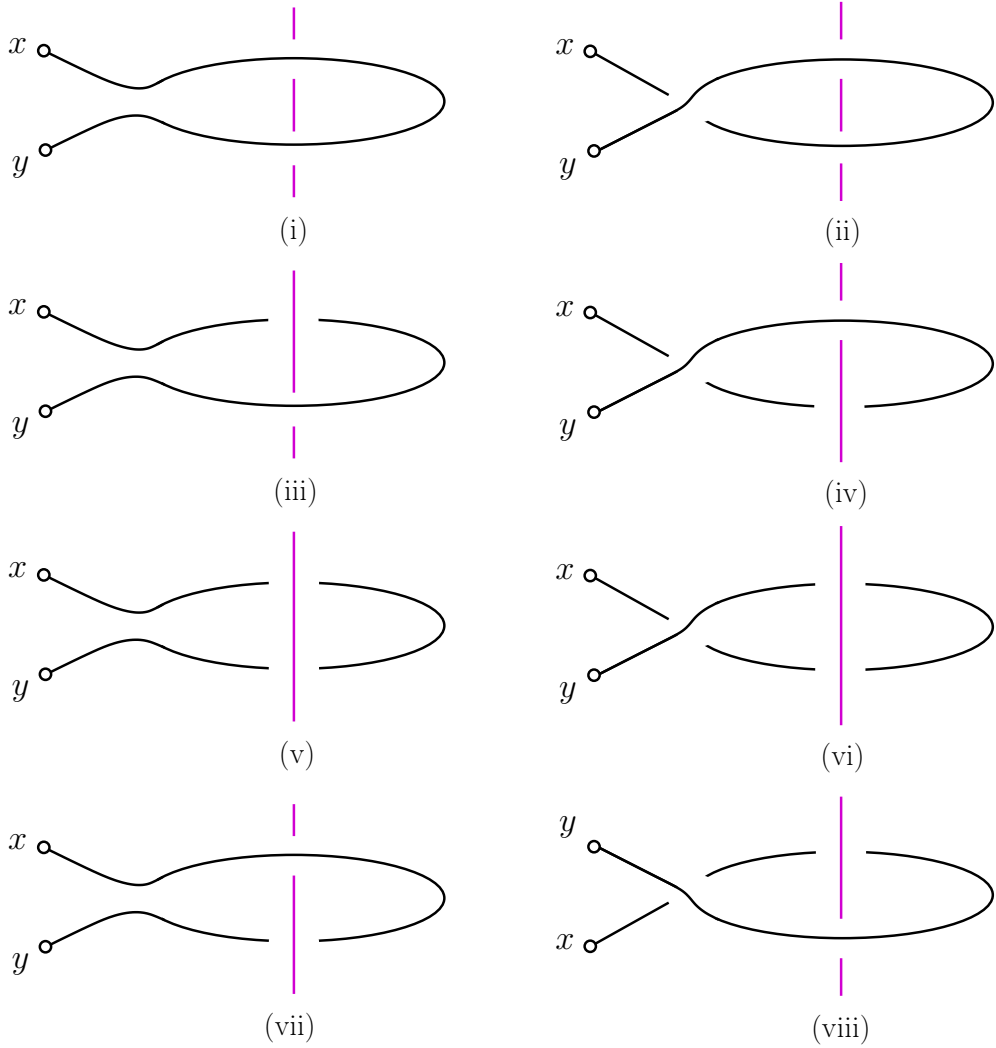}
\caption{Illustration of the proof of (R3) for the operation of simplyfing a $\Theta$.}
\label{fig:1890}
\end{figure}
\def\tf#1{{\Scale[2.4]{#1}}}
\def\tz#1{{\Scale[2.0]{#1}}}
\def\te#1{{\Scale[3.4]{#1}}}

As in the first case we discussed above, in each of these cases it is easy to see that the strand from $x$ to $y$  in $D$ can be isotoped to the strand from $x$ to $y$ in $\oD$. Since all the other vertices of $P$ are resolved in the same way as they are resolved in $\oP$, we conclude that the resulting diagram $D$ of $P$ is equivalent to the diagram $\oD$ of $\oP$. This completes the proof that the reduction operation of simplifying a $\Theta$ satisfies Property (R3).

\section{Irreducible projections and proof of Theorem~\ref{thm:main}}\label{sec:proofmain}

We say that a projection is {\em irreducible} if it is pairwise crossing and it is not possible to apply any reduction operation to it. The heart of the proof of Theorem~\ref{thm:main} is the following statement.

\begin{proposition}\label{pro:twoirr}
Up to equivalence, the only irreducible projections are the projections $P_1$ and $P_2$ in Figure~\ref{fig:670}.
\end{proposition}

\def\tf#1{{\Scale[2.8]{#1}}}
\begin{figure}[ht!]
\centering
\scalebox{0.28}{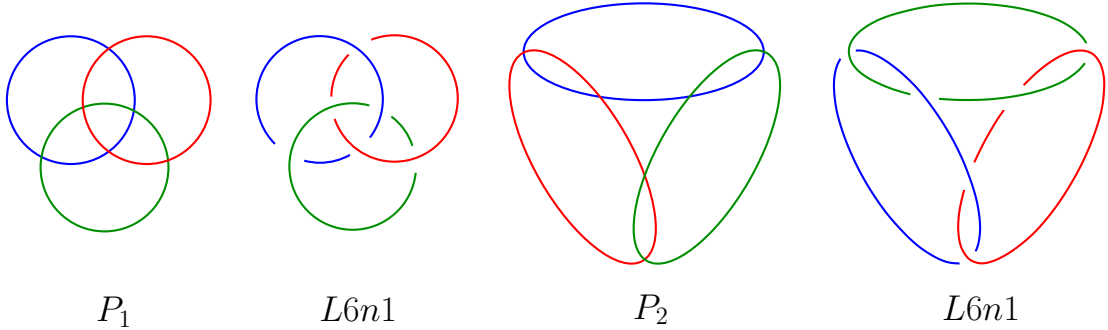}
\caption{The two irreducible projections $P_1$ and $P_2$.}
\label{fig:670}
\end{figure}
\def\tf#1{{\Scale[2.4]{#1}}}

The rest of the paper is devoted to the proof of Proposition~\ref{pro:twoirr}. As we will now see, Theorem~\ref{thm:main} follows easily from this statement.


\begin{proof}[Proof of Theorem~\ref{thm:main}, assuming Proposition~\ref{pro:twoirr}]
As we pointed out before stating Theorem~\ref{thm:main}, the ``only if'' part of the theorem holds simply because $L6n1$ is pairwise linked. Thus in order to prove the theorem we need to show the ``if'' part: every pairwise crossing projection is a projection of $L6n1$.

Let $P$ be any pairwise crossing projection. Our goal is to show that $L6n1\in \Ls{P}$. In view of Properties (R1) and (R2) of the reduction operations, given a pairwise crossing projection $P$, we can iteratively apply to $P$ a sequence of reduction operations until we reach an irreducible projection $P'$. An iterative application of (R3) then implies that $\Ls{P'}\subseteq \Ls{P}$. Since by Proposition~\ref{pro:twoirr} the only irreducible projections are $P_1$ and $P_2$, we conclude that ($\dag$) {\em either $\Ls{P_1} \subseteq \Ls{P}$ or $\Ls{P_2} \subseteq \Ls{P}$.}

As we illustrate in Figure~\ref{fig:670}, both $P_1$ and $P_2$ can be resolved into $L6n1$. That is, $L6n1\in\Ls{P_1}$ and $L6n1\in\Ls{P_2}$. In view of ($\dag$), it follows that $L6n1\in\Ls{P}$. 
\end{proof}

\section{Towards the proof of Proposition~\ref{pro:twoirr}: properties of irreducible projections}\label{sec:proofmain}

In this section we pave the way towards the proof of Proposition~\ref{pro:twoirr}, by establishing several properties that must be satisfied in an irreducible projection. More specifically, we identify two structures that cannot exist in an irreducible projection, and we show that every face in an irreducible projection must be bounded by a cycle.

\subsection{Disposable digons}

Let $P=B\cup R\cup G$ be a pairwise crossing projection. Suppose that there exist parallel edges $e_1,e_2$ such that $e_1\cup e_2$ bounds an open disk $\Delta$ that does not contain any part of $P$. We say that $e_1\cup e_2$ is a {\em digon}. See Figure~\ref{fig:700}.

\def\tf#1{{\Scale[2.0]{#1}}}
\def\tz#1{{\Scale[1.4]{#1}}}
\begin{figure}[ht!]
\centering
\scalebox{0.5}{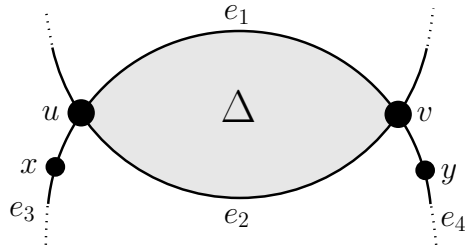}
\caption{A digon in a link projection. If either $u$ and $v$ are monochromatic, or they are bichromatic (necessarily of the same type) and they are not the only bichromatic vertices of their type, then we can shortcut $P$ at $x$ and $y$.}
\label{fig:700}
\end{figure}
\def\tf#1{{\Scale[2.4]{#1}}}
\def\tz#1{{\Scale[2.0]{#1}}}

Let $u,v$ be the common endvertices of $e_1$ and $e_2$. If $e_1$ and $e_2$ are of the same colour, then $u$ and $v$ are monochromatic of the same colour. Otherwise, $u$ and $v$ are bichromatic of the same type. We say that $e_1\cup e_2$ is a {\em disposable} digon if either (i)  $u$ and $v$ are monochromatic; or (ii) $u$ and $v$ are bichromatic and they are not the only bichromatic vertices of their type in $P$.

\begin{observation}\label{obs:disdig}
In an irreducible projection there cannot be any disposable digons.
\end{observation}

\begin{proof}
Let $e_1\cup e_2$ be a disposable digon in an irreducible projection $P$. As we illustrate in Figure~\ref{fig:700}, let $e_3$ (respectively, $e_4$) be the edge that precedes (respectively, succeeds) $e_1$ as we traverse the straight-ahead walk (either $B,R$, or $G$) that contains $e_1$. As we also illustrate in that figure, we subdivide $e_3$ (respectively, $e_4$) with a degree $2$ vertex $x$ (respectively, $y$).

Note that $x$ and $y$ are incident with the same face. One of the two $xy$-walks contains the vertices $u$ and $v$, and it does not contain any other vertices of $P$. Now the assumption that $e_1\cup e_2$ is disposable guarantees that the other $xy$-walk is colourful, and so we can shortcut $P$ at $x$ and $y$. But this contradicts the assumption that $P$ is irreducible.
\end{proof}

\subsection{Superfluous walks}

Let $P=B\cup R\cup G$ be a pairwise crossing projection, and suppose that $P$ contains a monochromatic vertex $v$. Without loss of generality, for the purposes of this discussion we may assume that $v$ is blue.

The straight-ahead blue closed walk $B$ is then the edge-disjoint union of two straight-ahead closed walks $\beta_1$ and $\beta_2$ that start and end at $v$. These are the $v$-{\em walks}. Similarly as when we defined the shortcut operation, for $i\in\{1,2\}$ we say that $\beta_i$ is {\em colourful} if it has at least one blue-red vertex and at least one blue-green vertex. If $\beta_1$ (respectively, $\beta_2$) is colourful, then we say that $\beta_2$ (respectively, $\beta_1$) is {\em superfluous}. 

\begin{observation}\label{obs:supwal}
In an irreducible projection there cannot be any superfluous walks.
\end{observation}

\begin{proof}
Using the notation and terminology from the previous discussion, by way of contradiction suppose that $\beta_2$ is superfluous, and so $\beta_1$ is colourful. As we illustrate in Figure~\ref{fig:1910}, we let $e_1,e_2$ be the edges incident with $v$ that are in $\beta_2$.

\def\tf#1{{\Scale[2.0]{#1}}}
\def\tz#1{{\Scale[1.5]{#1}}}
\begin{figure}[ht!]
\centering
\scalebox{0.45}{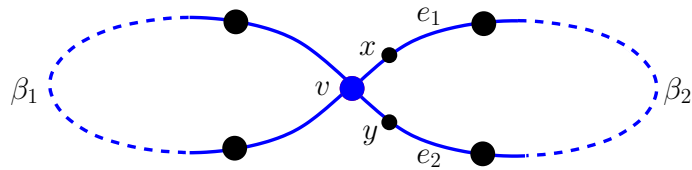}
\caption{Illustration of the proof of Observation~\ref{obs:supwal}.}
\label{fig:1910}
\end{figure}
\def\tf#1{{\Scale[2.4]{#1}}}
\def\tz#1{{\Scale[2.0]{#1}}}

As we also illustrate in Figure~\ref{fig:1910}, we subdivide $e_1$ (respectively, $e_2$) with a degree $2$ vertex $x$ (respectively, $y$). One of the two $xy$-walks contains all the edges and vertices of $\beta_1$. Since $\beta_1$ is colourful, it follows that this $xy$-walk is colourful, and so we can shortcut $P$ at $x$ and $y$. This contradicts the assumption that $P$ is irreducible.
\end{proof}

\subsection{Faces in irreducible projections}

We conclude this section with an important remark on irreducible projections.

\begin{observation}\label{obs:facial}
If $P$ is an irreducible projection, then every face of $P$ is bounded by a cycle.
\end{observation}

\begin{proof}
We show that an irreducible projection $P$ cannot have a cut vertex, that is, a vertex $v$ such that $P\setminus\{v\}$ is disconnected. This implies the observation, since it follows that every irreducible projection is $2$-connected, and in every spherical embedding of a $2$-connected graph each face is bounded by a cycle.

By way of contradiction, suppose that $P$ is an irreducible projection with a cut vertex $v$. It is easy to see that $v$ is necessarily monochromatic, and without loss of generality we may assume that $v$ is blue. Let $\beta_1,\beta_2$ be the $v$-walks. We note that one of $\beta_1$ and $\beta_2$ must contain all the blue-red vertices and all the blue-green vertices. Indeed, since $v$ is a cut-vertex, otherwise the red straight-ahead walk $R$ and the green straight-ahead walk $G$ would be disjoint, contradicting that $P$ is pairwise crossing.

Without loss of generality we may assume that $\beta_1$ contains all the blue-red vertices and all the blue-green vertices. Thus $\beta_1$ is colourful, and $\beta_2$ is superfluous. In view of Observation~\ref{obs:supwal}, this contradicts the irreducibility of $P$.
\end{proof}


\section{Good sections in pairwise crossing projections}\label{sec:goodsections}

The proof of Proposition~\ref{pro:twoirr} relies crucially on the concept of a good section. We refer the reader to Figure~\ref{fig:2040} and its caption for an illustration of the upcoming notions. Let $P=B\cup R\cup G$ be a pairwise crossing projection. A cycle of $P$ is {\em facial} if it bounds a face. A {\em section} of a facial cycle $C$ is a path contained in $C$, all of whose edges are of the same colour, and that is maximal with respect to this property. 

\def\te#1{{\Scale[3.2]{#1}}}
\def\tf#1{{\Scale[2.8]{#1}}}
\def\tz#1{{\Scale[2.4]{#1}}}
\begin{figure}[ht!]
\centering
\scalebox{0.295}{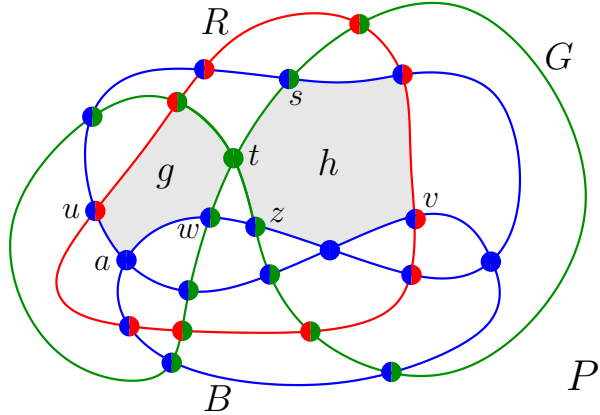}
\caption{The blue path $uaw$ is a section of the facial cycle that bounds face $g$. This section is good, as its endvertices are of distinct types: $u$ is blue-red and $w$ is blue-green. The green path $stz$ is a section of the facial cycle that bounds face $h$, but it is not a good section: its endvertices $s$ and $z$ are of the same type, namely green-blue.}
\label{fig:2040}
\end{figure}
\def\tf#1{{\Scale[2.4]{#1}}}
\def\tz#1{{\Scale[2.0]{#1}}}

Clearly, the endvertices of a section $S$ are necessarily bichromatic. If the endvertices of $S$ are of distinct types, then $S$ is a {\em good} section. For instance, in a good blue section one endvertex is blue-red, the other endvertex is blue-green, and all the internal vertices (if any) are blue.

The {\em length} of a section is its number of edges. Note that the length of a section may be one, as it may consist of only one edge and its endvertices.

The next statement plays a crucial role in the proof of Proposition~\ref{pro:twoirr}.

\begin{observation}\label{obs:goodsectionsexist}
Every pairwise crossing projection has at least two good sections of each colour.
\end{observation}

\begin{proof}
Let $P=B\cup R\cup G$ be a pairwise crossing projection. By symmetry, it suffices to show that there are at least two good blue sections in $P$. We illustrate the proof using the projection $P$ in Figure~\ref{fig:2040}. As we illustrate in Figure~\ref{fig:2080}(i), we start by letting $P'$ be the plane graph obtained by removing from $P$ all the green edges and all the green vertices, but keeping all the green-blue or green-red vertices.



It is easy to see that since $P$ is pairwise crossing there must exist a face $f$ of $P'$ whose boundary walk $W$ contains at least one green-blue vertex and at least one green-red vertex. See Figure~\ref{fig:2080}(i). In particular, $W$ contains both blue and red edges, and so $W$ may be written as a concatenation $W=W_1 W_2 \cdots W_k$ of an even number $k$ of walks, where $W_1, W_3, \ldots, W_{k-1}$ contain only blue edges, and $W_2, W_4, \ldots, W_k$ contain only red edges. For instance, as we show in Figure~\ref{fig:2080}(ii), the boundary walk $W$ of the face $f$ in Figure~\ref{fig:2080}(i) is the concatenation of a (thick) blue walk $W_1$, a (thick) red walk $W_2$, a (thin) blue walk $W_3$, and a (thin) red walk $W_4$.

\def\te#1{{\Scale[3.0]{#1}}}
\def\tf#1{{\Scale[2.6]{#1}}}
\def\tz#1{{\Scale[2.2]{#1}}}
\begin{figure}[ht!]
\centering
\scalebox{0.295}{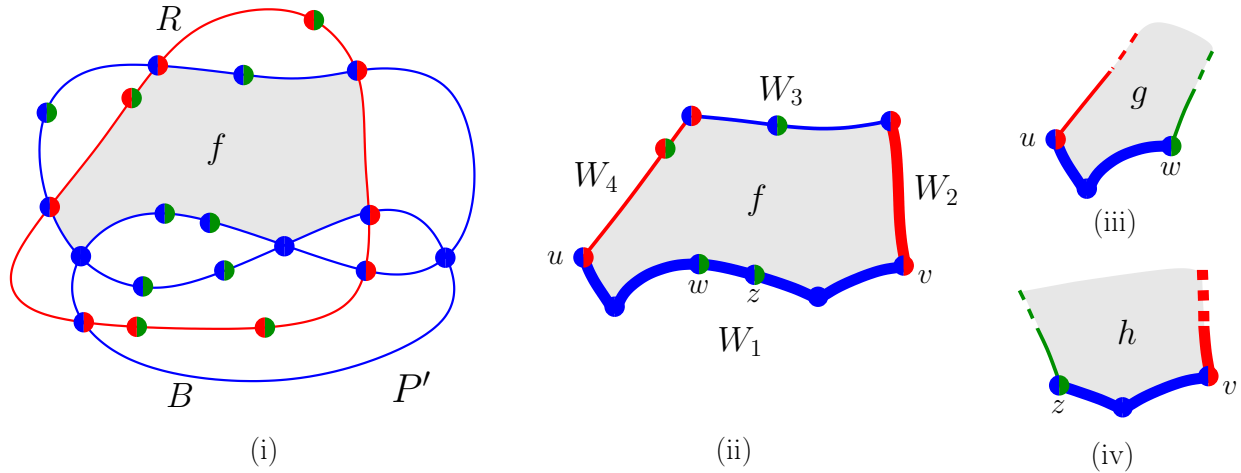}
\caption{Illustration of the proof of Observation~\ref{obs:goodsectionsexist}.}
\label{fig:2080}
\end{figure}
\def\tf#1{{\Scale[2.4]{#1}}}
\def\tz#1{{\Scale[2.0]{#1}}}

Since $W$ has at least one blue-green vertex, it follows that one of the blue walks $W_1, W_3, \ldots,$ $W_{k-1}$ contains at least one blue-green vertex. Without loss of generality, we may assume that $W_1$ has at least one blue-green vertex. Note that $W_1$ starts with a blue-red vertex $u$, ends with another blue-red vertex $v$, and every internal vertex of $W_1$ is either blue or blue-green. As we illustrate in Figure~\ref{fig:2080}(ii), we let $w$ (respectively, $z$) be the first (respectively, last) blue-green vertex that we encounter as we traverse $W_1$ from $u$ to $v$. We note that $w$ and $z$ may be the same vertex. 

As we illustrate in Figure~\ref{fig:2080}(iii), the subwalk of $W_1$ that starts at the blue-red vertex $u$ and ends at the blue-green vertex $w$ is a good section of a face $g$ back in $P$ (see also Figure~\ref{fig:2040}). Similarly, as we show in Figure~\ref{fig:2080}(iv), the subwalk of $W_1$ that starts at the blue-green vertex $z$ and ends at the blue-red vertex $v$ is a good section of a face $h$ back in $P$ (see also Figure~\ref{fig:2040}). Thus there are at least two good blue sections in $P$, and so we are done.
\end{proof}

\section{Proof of Proposition~\ref{pro:twoirr}}\label{sec:prooftwoirr}

The proof of Proposition~\ref{pro:twoirr} has two major ingredients, captured in the following two lemmas.

\begin{lemma}\label{lem:nomono}
If $P$ is an irreducible projection, then all good sections of $P$ have length one.
\end{lemma}

\begin{lemma}\label{lem:simple}
If $P$ is an irreducible projection in which all good sections have length one, then $P$ has exactly $6$ vertices.
\end{lemma}

Before moving on to the proofs of Lemmas~\ref{lem:nomono} and~\ref{lem:simple}, we note that Proposition~\ref{pro:twoirr} follows easily from them.

\begin{proof}[Proof of Proposition~\ref{pro:twoirr}, assuming Lemmas~\ref{lem:nomono} and~\ref{lem:simple}]
Let $P=B\cup R\cup G$ be an irreducible projection. Combining Lemmas~\ref{lem:nomono} and~\ref{lem:simple} we obtain that $P$ has exactly $6$ vertices. Since every pairwise projection has at least two bichromatic vertices of each type, it follows that $P$ has exactly two blue-red vertices, exactly two blue-green vertices, exactly two red-green vertices, and no monochromatic vertices of any colour. That is, $B,R$, and $G$ are cycles that pairwise cross each other in exactly two vertices. It is a straightforward exercise to verify that then $P$ is equivalent to either $P_1$ or $P_2$ in Figure~\ref{fig:670}: this is the well-known fact that every arrangement of three pseudocircles that pairwise cross exactly twice is equivalent to either the Krupp arrangement or to the non-Krupp arrangement (see ~\cite[Figure 1]{felsnerscheucher}).
\end{proof}

\subsection{Proof of Lemma~\ref{lem:nomono}}

Lemma~\ref{lem:nomono} is an immediate consequence of the next two statements.

\begin{claim}\label{cla:ge3}
In an irreducible projection there cannot be any good section of length at least three.
\end{claim}

\begin{claim}\label{cla:eq2}
In an irreducible projection there cannot be any good section of length exactly two.
\end{claim}

\begin{proof}[Proof of Claim~\ref{cla:ge3}]
By way of contradiction, suppose that $P=B\cup R\cup G$ has a facial cycle $C$ with a good section $S=v_0 e_1 v_1 \ldots e_k v_k$ of length $k\ge 3$. Recall that the endvertices of every good section are bichromatic vertices of distinct types. Without loss of generality we may assume that $S$ is blue, and that $v_0$ is blue-red and $v_k$ is blue-green. See Figure~\ref{fig:230} for an illustration.

\def\te#1{{\Scale[3.4]{#1}}}
\def\tf#1{{\Scale[3.4]{#1}}}
\def\tz#1{{\Scale[3.2]{#1}}}
\begin{figure}[ht!]
\centering
\scalebox{0.23}{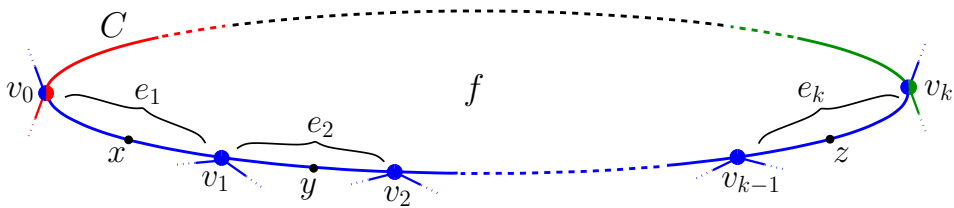}
\caption{Illustration of the proof of Claim~\ref{cla:ge3}.}
\label{fig:230}
\end{figure}
\def\tf#1{{\Scale[2.4]{#1}}}
\def\tz#1{{\Scale[2.0]{#1}}}

Subdivide $e_1$ with a degree $2$ vertex $x$, $e_2$ with a degree $2$ vertex $y$, and $e_k$ with a degree $2$ vertex $z$. To prove the claim we show that either (i) at least one of the two $xy$-walks is colourful; or (ii) at least one of the two $yz$-walks is colourful; or (iii) at least one of the two $xz$-walks is colourful. This will complete the proof, since in either case we can apply a shortcut operation on $P$, contradicting its irreducibility.

We traverse the blue straight-ahead closed walk $B$ starting at $x$ with the blue half-edge $x v_1$. Suppose that in this traversal we encounter $y$ before we encounter $v_k$. In this case the part of the traversal from $x$ to $y$ (that is, one of the two $xy$-walks) clearly contains neither $v_0$ nor $v_k$. Thus the other $xy$-walk contains both $v_0$ and $v_k$, and so it is colourful. Therefore in this case (i) holds. A totally analogous argument shows that if in the traversal we encounter $z$ before we encounter $v_k$, then one of the two $xz$-walks is colourful, and so (ii) holds. 

We may then assume that in the traversal of $B$ starting with $x v_1$ we encounter first $v_k$, and then we encounter $y$ and $z$ in some order. In this case the part of the traversal that has $y$ and $z$ as endpoints is a $yz$-walk that contains neither $v_0$ nor $v_k$. Thus the other $yz$-walk contains both $v_0$ and $v_k$, and so it is colourful. Therefore in this case (iii) holds.
\end{proof}

\begin{proof}[Proof of Claim~\ref{cla:eq2}]
By way of contradiction, suppose that $P=B\cup R\cup G$ has a facial cycle $C$ with a good section $S=u e v e' w$ of length $2$. Without loss of generality we may assume that $S$ is blue, and that $u$ is blue-red and $w$ is blue-green. See Figure~\ref{fig:1140} for an illustration. As we also illustrate in that figure, we let $B_1$ denote the $v$-walk that contains $u$ and $e$, and we let $B_2$ denote the other $v$-walk, which is the one that contains $w$ and $e'$. 

\def\te#1{{\Scale[3.4]{#1}}}
\def\tf#1{{\Scale[2.8]{#1}}}
\def\tz#1{{\Scale[2.4]{#1}}}
\begin{figure}[ht!]
\centering
\scalebox{0.3}{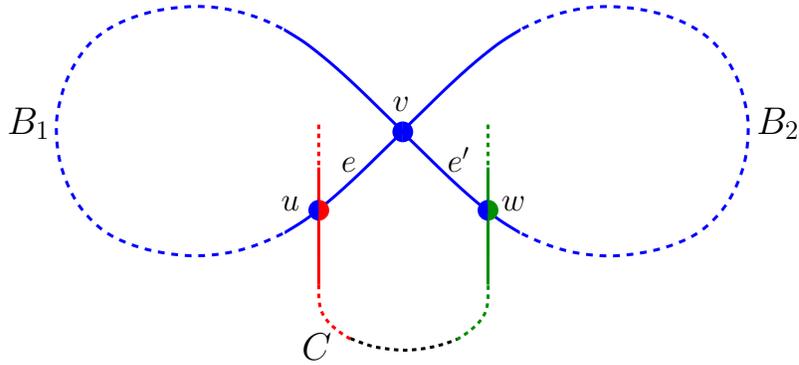}
\caption{Illustration of the proof of Claim~\ref{cla:eq2}.}
\label{fig:1140}
\end{figure}
\def\tf#1{{\Scale[2.4]{#1}}}
\def\tz#1{{\Scale[2.0]{#1}}}

Our strategy to prove the claim is to find structural properties of $P$ that follow from its irreducibility. These properties are captured as statements (I), (II), (III), (IV), and (V) below. As we argue after the proof of (V), these properties imply the existence of a $\Theta$ in $P$ (see Figure~\ref{fig:1540}). This will conclude the proof of the claim, since the existence of a $\Theta$ contradicts the irreducibility of $P$.

\vglue 0.2 cm
\noindent{(I) {\sl $B_1$ does not contain any blue-green vertex, and $B_2$ does not contain any blue-red vertex.}}

\begin{proof}
By way of contradiction, suppose that $B_1$ contains some blue-green vertex. Since $B_1$ contains also the blue-red vertex $u$, it follows that $B_1$ is a colourful $v$-walk, and so $B_2$ is a superfluous $v$-walk. In view of Observation~\ref{obs:supwal}, this contradicts the irreducibility of $P$. A totally analogous argument shows that $B_2$ cannot contain any blue-red vertex.
\end{proof}

\vglue 0.2 cm
\noindent{(II) {\sl $B_1$ is a cycle, and $B_2$ is a cycle.}}

\begin{proof}
By symmetry, it suffices to show that $B_2$ is a cycle. By way of contradiction, suppose that this is not the case. As we illustrate in Figure~\ref{fig:1290}, then there must exist a blue vertex $z$ in $B_2$ such that the four edges incident with $z$ are in $B_2$. As we also illustrate in that figure, the $z$-walk that contains $v$ also contains $u$ and $w$, and so it is colourful. Thus the other $z$-walk (the thick one) is superfluous. In view of Observation~\ref{obs:supwal}, this contradicts the irreducibility of $P$. 
\end{proof}

\def\te#1{{\Scale[3.4]{#1}}}
\def\tf#1{{\Scale[2.8]{#1}}}
\def\tz#1{{\Scale[2.4]{#1}}}
\begin{figure}[ht!]
\centering
\scalebox{0.3}{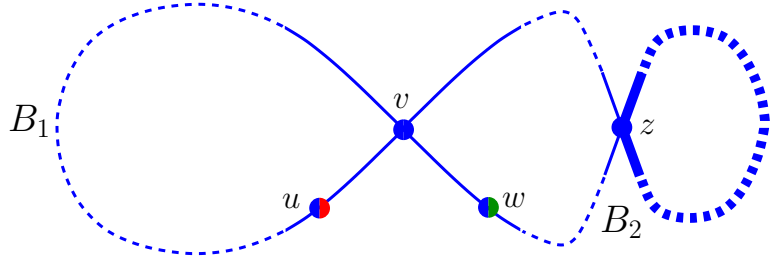}
\caption{Illustration of the proof of (II) in Claim~\ref{cla:eq2}.}
\label{fig:1290}
\end{figure}
\def\tf#1{{\Scale[2.4]{#1}}}
\def\tz#1{{\Scale[2.0]{#1}}}

Since $B_2$ is a cycle, by the Jordan curve theorem $\sphere\setminus B_2$ has two connected components, each of which is homeomorphic to an open disk. We let $\Delta$ denote the connected component of $\sphere\setminus B_2$ that does not contain the blue-red vertex $u$. We refer the reader to Figure~\ref{fig:1340} for an illustration, where $\Delta$ is the shaded region.

\vglue 0.2 cm
\noindent{(III) {\sl If $g=st$ is an edge of $P$ contained in $\Delta$, and its endvertices $s$ and $t$ are both in $B_2$, then either $s=w$ or $t=w$.}}

\begin{proof}
By way of contradiction, suppose that $s{\neq}w$ and $t{\neq}w$. See Figure~\ref{fig:1340}. As we also illustrate in that figure, we subdivide one of the two edges of $B_2$ incident with $s$ with a degree $2$ vertex $x$, and we subdivide one of the two edges of $B_2$ incident with $t$ with a degree $2$ vertex $y$. 

\def\te#1{{\Scale[3.4]{#1}}}
\def\tf#1{{\Scale[2.8]{#1}}}
\def\tz#1{{\Scale[2.4]{#1}}}
\begin{figure}[ht!]
\centering
\scalebox{0.28}{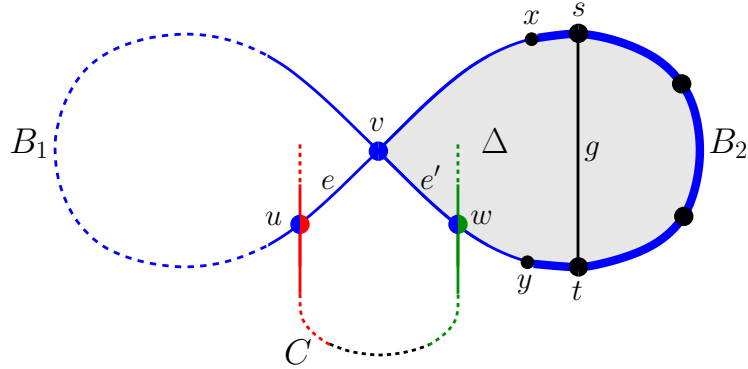}
\caption{Illustration of the proof of (III) in Claim~\ref{cla:eq2}.}
\label{fig:1340}
\end{figure}
\def\tf#1{{\Scale[2.4]{#1}}}
\def\tz#1{{\Scale[2.0]{#1}}}

Since $g$ is an edge, then $x$ and $y$ are in the same face of $P$. Now one of the $xy$-walks is completely contained in $B_2$: this is the $xy$-walk highlighted with thick edges in Figure~\ref{fig:1340}. The other $xy$-walk contains $u$ and $w$, and so it is colourful. Since $x$ and $y$ are in the same face it follows that we can shortcut $P$ at $x$ and $y$. But this contradicts the irreducibility of $P$.
\end{proof}

\vglue 0.2 cm
\noindent{(IV) {\sl The only blue vertex is $v$.}}
\vglue 0.2 cm

\begin{proof}
By way of contradiction, suppose that there is a blue vertex $s$ distinct from $v$. Since each of $B_1$ and $B_2$ is a cycle, it follows that $s$ must be in $B_1\cap B_2$. As we illustrate in Figure~\ref{fig:1410}, it follows that there must be a subpath $Q$ of $B_1$, contained in $\Delta$, that has $s$ as one of its endvertices. We let $t$ denote the endvertex of $Q$ that is not $s$.

\def\te#1{{\Scale[3.4]{#1}}}
\def\tf#1{{\Scale[2.8]{#1}}}
\def\tz#1{{\Scale[2.4]{#1}}}
\begin{figure}[ht!]
\centering
\scalebox{0.28}{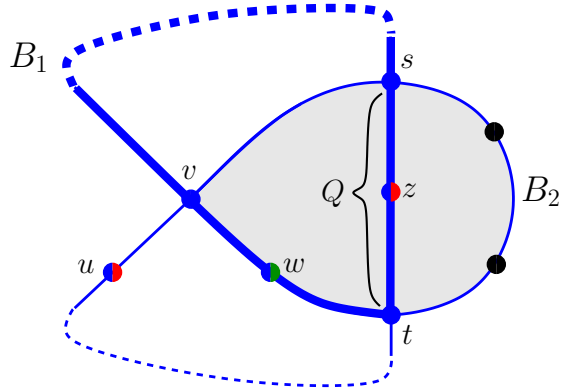}
\caption{Illustration of the proof of (IV) in Claim~\ref{cla:eq2}.}
\label{fig:1410}
\end{figure}
\def\tf#1{{\Scale[2.4]{#1}}}
\def\tz#1{{\Scale[2.0]{#1}}}

It follows from (III) that $Q$ cannot consist of a single edge, and so there is at least one internal vertex $z$ of $Q$. We claim that $z$ must be blue-red, as illustrated in Figure~\ref{fig:1410}. To see this, first we note that $z$ cannot be blue, since every blue vertex is in $B_1\cap B_2$, and hence in the boundary of $\Delta$. Thus $z$ is either blue-red or blue-green. Now if $z$ is blue-green, then the $v$-walk $B_1$ is colourful (as it also contains the blue-red vertex $u$), and so $B_2$ is superfluous. Since by Observation~\ref{obs:supwal} this contradicts the irreducibility of $P$, we conclude that $z$ cannot be blue-green. Thus $z$ is blue-red, as claimed.

To finish the proof we note that one of the two $t$-walks contains $w,v,s$, and $z$: this is the $t$-walk highlighted with thick edges in Figure~\ref{fig:1410}. Since $w$ is blue-green and $z$ is blue-red, this $t$-walk is colourful. Thus the other $t$-walk is superfluous. In view of Observation~\ref{obs:supwal}, this contradicts the irreducibility of $P$.
\end{proof}

\vglue 0.2 cm
\noindent{(V) {\sl There is no green vertex contained in $\Delta$.}}
\vglue 0.2 cm

\begin{proof}
By way of contradiction, suppose that there is a green vertex $z$ contained in $\Delta$. We start by noting that the red straight-ahead walk $R$ lies entirely outside $\Delta$. Indeed, if some part of $R$ were contained in $\Delta$, since the blue-red vertex $u$ is outside $\Delta$ then necessarily $R$ would intersect the boundary $B_2$ of $\Delta$. In particular, $B_2$ would contain at least one blue-red vertex, contradicting (I).

Let $G_1, G_2$ be the two green $z$-walks. Since $P$ is pairwise crossing it follows that at least one of $G_1$ and $G_2$ has a green-red vertex. Without loss of generality, we may assume that $G_1$ has a red-green vertex. Since $R$ lies entirely outside $\Delta$, and $z$ is inside $\Delta$, it follows that $G_1$ must cross $B_2$. In particular, $G_1$ has at least one green-blue vertex. Thus $G_1$ has at least one green-red vertex and at least one green-blue vertex, and so it is colourful. Thus $G_2$ is a superfluous $z$-walk. By Observation~\ref{obs:supwal}, this contradicts the irreducibility of $P$.
\end{proof}

\vglue 0.3cm

\noindent{\em Conclusion of the proof of Claim~\ref{cla:eq2}.} We start by noting that no vertex can be contained in $\Delta$. Indeed, by way of contradiction, suppose that some vertex $z$ of $P$ is contained in $\Delta$. As we argued in the proof of (V), the red straight-ahead walk $R$ is contained outside $\Delta$, and so $z$ is neither red, nor red-blue, nor red-green. We know from (IV) that $v$ is the only blue vertex, and from this it follows that no blue edge can be inside $\Delta$. Therefore $z$ is neither blue, nor blue-red, nor blue-green. The only possibility left is that $z$ is green, but this is ruled out using (V). We conclude that no vertex is contained in $\Delta$.

Let $e_w$ be the green edge incident with the blue-green vertex $w$ that is contained in $\Delta$. Since no vertex is contained in $\Delta$, it follows that the other endvertex of $e_w$ is a blue-green vertex $q$ in $B_2$. See Figure~\ref{fig:1540}. We claim that $e_w$ is the only part of $P$ contained in $\Delta$. Indeed, by way of contradiction suppose that some edge $g=st$ other than $e_w$ is contained in $\Delta$. Since no vertex is contained in $\Delta$, it follows that both $s$ and $t$ are in $B_2$. Since $g\neq e_w$, it follows that neither $s$ nor $t$ is equal to $w$, contradicting (III). 

\def\te#1{{\Scale[3.4]{#1}}}
\def\tf#1{{\Scale[2.8]{#1}}}
\def\tz#1{{\Scale[2.4]{#1}}}
\begin{figure}[ht!]
\centering
\scalebox{0.3}{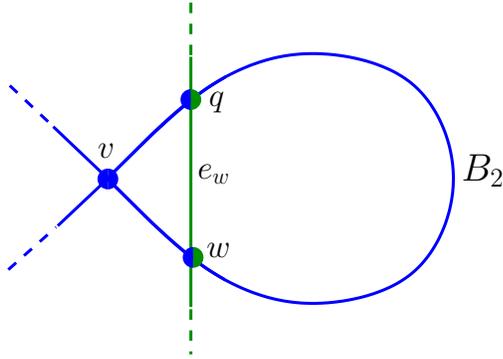}
\caption{Conclusion of the proof of Claim~\ref{cla:eq2}.}
\label{fig:1540}
\end{figure}
\def\tf#1{{\Scale[2.4]{#1}}}
\def\tz#1{{\Scale[2.0]{#1}}}

Thus $e_w$ is the only part of $P$ contained in $\Delta$. As illustrated in Figure~\ref{fig:1540}, it follows that $v,q$, and $w$ are the only vertices of $B_2$, and so $B_2 + e_w$ is a $\Theta$ in $P$. But this is impossible, since by assumption $P$ is irreducible.
\end{proof}

\subsection{Proof of Lemma~\ref{lem:simple}}
As we shall see, Lemma~\ref{lem:simple} follows from three easy observations and an application of Euler's formula.

\begin{proof}[Proof of Lemma~\ref{lem:simple}]
Let $P=B\cup R\cup G$ be an irreducible projection in which all good sections have length one. That is, every good section is an edge, which we will call a {\em good} edge. Thus, for instance, a blue edge is good if and only if one of its endvertices is blue-red and the other is blue-green. We recall from Observation~\ref{obs:goodsectionsexist} that there exist at least two good sections of each colour, and so it follows that $P$ has at least two good edges of each colour.

\vglue 0.2 cm
\noindent{\bf Claim A.} {\sl $P$ does not have any monochromatic vertices.}
\vglue 0.2 cm

\begin{proof}
By symmetry it suffices to show that there are no blue vertices in $P$. By way of contradiction, suppose that $P$ has a blue vertex $u$. Let $e$ be a good blue edge. One of the two $u$-walks contains the edge $e$, and so it contains at least one blue-red vertex and at least one blue-green vertex. Therefore the $u$-walk that does not contain $e$ is superfluous. In view of Observation~\ref{obs:supwal}, this contradicts the assumption that $P$ is irreducible. 
\end{proof}

The {\em degree} $|f|$ of a face $f$ of $P$ is the number of edges in the facial cycle that bounds $f$.

\vglue 0.2 cm
\noindent{\bf Claim B.} {\sl If $P$ has at least three faces of degree $2$, then $P$ has exactly $6$ vertices.}
\vglue 0.2 cm

\begin{proof}
By way of contradiction suppose that $P$ has at least $8$ vertices, and that there exist three faces $f_1, f_2, f_3$ of $P$ that have degree $2$. Then for $i=1,2,3$ the face $f_i$ is bounded by a digon $D_i$. By Claim A we have that $P$ has no monochromatic vertices, and so each of $D_1,D_2$, and $D_3$ is a digon whose two vertices are bichromatic of the same type. 

Without loss of generality, we may assume that the vertices of $D_1$ are blue-red. In view of Observation~\ref{obs:disdig} it follows that these are the two only blue-red vertices of $P$, as otherwise $D_1$ would be a disposable digon. Once again without loss of generality we may then assume that the vertices of $D_2$ are blue-green, and again using Observation~\ref{obs:disdig} it follows that these two are the only blue-green vertices of $P$. Therefore the vertices of $D_3$ are necessarily red-green, and invoking again Observation~\ref{obs:disdig} these two are the only red-green vertices of $P$. Combining these conclusions we have that $P$ has exactly two blue-red vertices, two blue-green vertices, and two red-green vertices. Since all the vertices of $P$ are bichromatic, it follows that $P$ has exactly $6$ vertices.
\end{proof}

\vglue 0.2 cm
\noindent{\bf Claim C.} {\sl $P$ has at most three faces of degree at least $4$.}
\vglue 0.2 cm

\begin{proof}
By way of contradiction, suppose that $P$ has at least four faces of degree at least $4$. Clearly, the facial cycle of a face of degree at least $4$ contains (at least) two edges of the same colour. Thus there exists a colour $c\in\{\text{\rm blue,red,green}\}$ and two faces $f,f'$ of $P$ such that the facial cycle $C$ of $f$ contains at least two edges of colour $c$, and also the facial cycle $C'$ of $f'$ contains at least two edges of colour $c$. Without loss of generality, we may assume that $c$ is blue.

Therefore $C$ contains two blue edges $e_1,e_2$, and $C'$ contains two blue edges $e_3,e_4$. Note that $\{e_1,e_2\}$ and $\{e_3,e_4\}$ are not necessarily disjoint, but they cannot be the same set. Since $P$ has at least two good blue edges, this implies that either (i) there is a good blue edge $e$ that is not in $\{e_1,e_2\}$; or (ii) there is a good blue edge $e$ that is not in $\{e_3,e_4\}$. Without loss of generality, we may assume that (i) holds.

Recall that $e_1$ and $e_2$ are blue edges in the boundary of the same face $f$. We subdivide $e_1$ (respectively, $e_2$) with a degree $2$ vertex $x$ (respectively, $y$). One of the two $xy$-walks contains the edge $e$. Since $e$ is good, one of its endvertices is blue-red and the other one is blue-green. Therefore one of the two $xy$-walks is colourful, and so we can shortcut $P$ at $x$ and $y$. But this contradicts the assumption that $P$ is irreducible.
\end{proof}

We know from Claim A that $P$ has no monochromatic vertices, and so each of $B,R$, and $G$ is a cycle. Any two of these cycles have an even number of vertices in common. That is, there is an even number of bichromatic vertices of each type. Therefore $P$ has an even number of vertices. Since Lemma~\ref{lem:simple} claims that $P$ has exactly six vertices, in order to prove the lemma we need to show that $P$ cannot have $8$ or more vertices.

By way of contradiction, suppose that $P$ has $n\ge 8$ vertices. Since $P$ is a $4$-regular graph, it follows that $P$ has $2n$ edges. Using Euler's formula we obtain that $P$ has $n+2$ faces, which we label $f_1,f_2,\ldots,f_{n+2}$ so that $|f_1|\le |f_2| \le \cdots \le |f_{n+2}|$. Since $\sum_{i=1}^{n+2} |f_i|$ equals twice the number of edges of $P$, it follows that $\sum_{i=1}^{n+2} |f_i| = 4n$.

By Claim B we have that $|f_3|\ge 3$, and Claim C implies that $|f_{n-1}|\le 3$. Combining these observations we have that $|f_3|=|f_4|=\cdots=|f_{n-1}|=3$. Thus there are at least $n-3$ faces of degree $3$. Since $n\ge 8$, there are at least $5$ faces of degree $3$. Note that since there are no monochromatic vertices, every face of degree $3$ consists of a good blue edge, a good red edge, and a good green edge. Since every edge belongs to exactly two faces, and there are at least $5$ faces of degree $3$, we conclude that there are at least three good edges of each colour.

We claim that this implies that there cannot be any face of degree at least four. Indeed, seeking a contradiction, suppose that there is a face $f$ of degree at least four. Then $f$ has (at least) two edges $e',e''$ of the same colour, which without loss of generality we may assume to be blue. We subdivide $e'$ (respectively, $e''$) with a degree $2$ vertex $x$ (respectively, $y$). Since there are at least three good blue edges, there exists a good blue edge $e$ that is neither $e'$ nor $e''$. Since one endvertex of $e$ is blue-red and the other is blue-green, the $xy$-walk that contains $e$ is colourful. Thus we can shortcut $P$ at $x$ and $y$, contradicting the irreducibility of $P$.

Thus no face has degree at least four, and so necessarily $|f_{n}|=|f_{n+1}|=|f_{n+2}|=3$. Therefore $\sum_{i=1}^{n+2}|f_i| = |f_1| + |f_2| + \sum_{i=3}^{n+2}|f_i| \le 3 + 3 + 3\cdot n= 3n+6$. Since $\sum_{i=1}^{n+2} |f_i| = 4n$, it follows that $4n\le 3n+6$, and so $n\le 6$. This contradicts the assumption that $n\ge 8$.
\end{proof}



\section{Concluding remarks}

Taniyama's motivation in~\cite{taniyamaknots} was the investigation of a relation $\ge$ on links. Following Taniyama, given two links $L_1,L_2$, we write $L_1 \ge L_2$ if every projection of $L_1$ is also a projection of $L_2$. In this case we say that $L_2$ is a {\em minor} of $L_1$, and that $L_1$ {\em majorizes} $L_2$.


The characterization given by Theorem~\ref{thm:main} determines which 3-component links majorize $L6n1$: {\em a 3-component link $L$ majorizes $L6n1$ if and only if every projection of $L$ is pairwise crossing}. Clearly, a $3$-component link $L$ satisfies that all its projections are pairwise crossing if and only if its components are pairwise linked. Thus we have the following.

\begin{observation}\label{obs:maj2}
A link $L$ majorizes $L6n1$ if and only if $L$ is a $3$-component link whose components are pairwise linked.
\end{observation}

In the other direction, Theorem~\ref{thm:main} can also be used to find out which links are majorized by $L6n1$: {\em a $3$-component link $L$ is majorized by $L6n1$ if and only if every pairwise crossing projection $P$ is a projection of $L$}. It is not difficult to show that the only links (other than $L6n1$) that satisfy this property are the three links in Figure~\ref{fig:850}. Therefore:

\begin{observation}\label{obs:maj1}
The links that are majorized by $L6n1$ are $L6n1$ itself and the links in Figure~\ref{fig:850}. 
\end{observation}

\def\tz#1{{\Scale[3.0]{#1}}}
\def\tf#1{{\Scale[1.6]{#1}}}
\begin{figure}[ht!]
\centering
\scalebox{0.62}{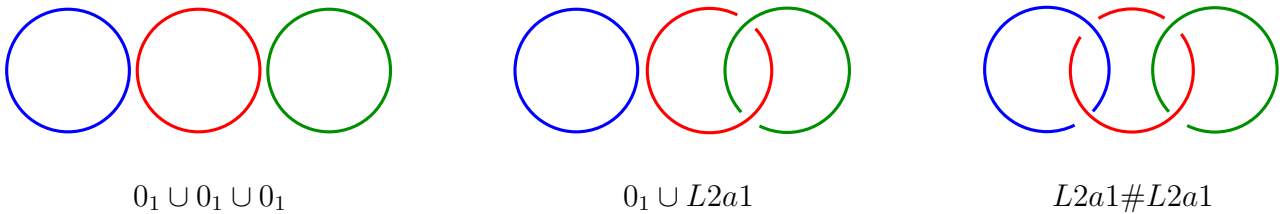}
\caption{Links that are majorized by $L6n1$.}
\label{fig:850}
\end{figure}
\def\tz#1{{\Scale[2.0]{#1}}}
\def\tf#1{{\Scale[2.4]{#1}}}

\section*{Acknowledgments}

This work was supported by CONACYT under Proyecto Ciencia de Frontera 191952.

\bibliographystyle{abbrv}
\bibliography{refs.bib}

\begin{thebibliography}{10}

\bibitem{atlas}
{The Thistlethwaite Link Table}.
\newblock \url{http://katlas.org/wiki/The_Thistlethwaite_Link_Table}.
\newblock Accessed: 2022-08-25.

\bibitem{cantarella}
J.~Cantarella, A.~Henrich, E.~Magness, O.~O'Keefe, K.~Perez, E.~Rawdon, and
  B.~Zimmer.
\newblock Knot fertility and lineage.
\newblock {\em J. Knot Theory Ramifications}, 26(13):1750093, 20, 2017.

\bibitem{endoitoh}
T.~Endo, T.~Itoh, and K.~Taniyama.
\newblock A graph-theoretic approach to a partial order of knots and links.
\newblock {\em Topology Appl.}, 157(6):1002--1010, 2010.

\bibitem{evenzohar}
C.~Even-Zohar, J.~Hass, N.~Linial, and T.~Nowik.
\newblock Universal knot diagrams.
\newblock {\em J. Knot Theory Ramifications}, 28(7):1950031, 30, 2019.

\bibitem{felsnerscheucher}
S.~Felsner and M.~Scheucher.
\newblock Arrangements of pseudocircles: On circularizability.
\newblock {\em Discrete Comput. Geom.}, 64(3):776–813, 2020.

\bibitem{hanaki2009}
R.~Hanaki.
\newblock Regular projections of knotted double-handcuff graphs.
\newblock {\em J. Knot Theory Ramifications}, 18(11):1475--1492, 2009.

\bibitem{hanaki2010}
R.~Hanaki.
\newblock Pseudo diagrams of knots, links and spatial graphs.
\newblock {\em Osaka J. Math.}, 47(3):863--883, 2010.

\bibitem{hanaki2014}
R.~Hanaki.
\newblock Trivializing number of knots.
\newblock {\em J. Math. Soc. Japan}, 66(2):435--447, 2014.

\bibitem{hanaki2015}
R.~Hanaki.
\newblock On scannable properties of the original knot from a knot shadow.
\newblock {\em Topology Appl.}, 194:296--305, 2015.

\bibitem{hanaki2020}
R.~Hanaki.
\newblock On fertility of knot shadows.
\newblock {\em J. Knot Theory Ramifications}, 29(11):2050080, 6, 2020.

\bibitem{huhtaniyama}
Y.~Huh and K.~Taniyama.
\newblock Identifiable projections of spatial graphs.
\newblock {\em J. Knot Theory Ramifications}, 13(8):991--998, 2004.

\bibitem{itotakimura}
N.~Ito and Y.~Takimura.
\newblock {$(1,2)$} and weak {$(1,3)$} homotopies on knot projections.
\newblock {\em J. Knot Theory Ramifications}, 22(14):1350085, 14, 2013.

\bibitem{medina1}
C.~Medina and G.~Salazar.
\newblock The knots that lie above all shadows.
\newblock {\em Topology Appl.}, 268:106922, 13, 2019.

\bibitem{smooth}
C.~Medina and G.~Salazar.
\newblock When can a link be obtained from another using crossing exchanges and
  smoothings?
\newblock {\em Topology Appl.}, 260:13--22, 2019.

\bibitem{ptaniyama}
J.~H. Przytycki and K.~Taniyama.
\newblock Almost positive links have negative signature.
\newblock {\em J. Knot Theory Ramifications}, 19(2):187--289, 2010.

\bibitem{takimura2018}
Y.~Takimura.
\newblock Regular projections of the knot {$6_2$}.
\newblock {\em J. Knot Theory Ramifications}, 27(14):1850081, 31, 2018.

\bibitem{taniyamaknots}
K.~Taniyama.
\newblock A partial order of knots.
\newblock {\em Tokyo J. Math.}, 12(1):205--229, 1989.

\bibitem{taniyamalinks}
K.~Taniyama.
\newblock A partial order of links.
\newblock {\em Tokyo J. Math.}, 12(2):475--484, 1989.

\end{thebibliography}

\end{document}